\documentclass[a4paper,reqno]{amsart}

\usepackage{amssymb}
\usepackage{latexsym}
\usepackage{amsmath}
\usepackage{euscript}
\usepackage{bbm}
\usepackage{tikz}
\usetikzlibrary{hobby,backgrounds,patterns}

\def\cal H{{\mathcal H}}

\def\phi{\varphi}

\renewcommand{\theta}{\vartheta}

\newtheorem{theorem}{Theorem}[section]
\newtheorem*{thm*}{Theorem}

\theoremstyle{definition}

\numberwithin{equation}{section}

\DeclareMathOperator{\diam}{diam}

\begin{document}

\title[On the hot spots of quantum trees]{On the hot spots of quantum trees}

%% Please delete not needed author entries.
%% Information for the first author.
\author{James Kennedy}
\address{Grupo de F\'isica Matem\'atica, Faculdade de Ci\^{e}ncias, Universidade de Lisboa, Campo Grande, Edif\'icio C6, P-1749-016 Lisboa, Portugal}
\email{jbkennedy@fc.ul.pt}
%%
%%    Information for the second author
\author{Jonathan Rohleder}%
\address{Matematiska institutionen, Stockholms universitet, 106 91 Stockholm, Sweden}
\email{jonathan.rohleder@math.su.se}

%%
%%    \dedicatory{This is a dedicatory.}
%%
%%    Abstract is required.

\begin{abstract}
We show that any second eigenfunction of the Laplacian with standard vertex conditions on a metric tree graph attains its extremal values only at degree one vertices, and give an example where these vertices do not realise the diameter of the graph.
\end{abstract}
%% maketitle must follow the abstract.
\maketitle                   % Produces the title.

\section{Introduction and preliminaries}

The hot spots conjecture asserts that the eigenfunction of the first positive eigenvalue of the Neumann Laplacian on a bounded domain $\Omega \subset \mathbb{R}^d$ should reach its maximum and minimum (only) at the boundary $\partial\Omega$. The intuition is that these points should be located as ``far away'' from each other as possible in some appropriate, weighted sense. While there are counterexamples to the conjecture in full generality, it is still open for convex domains. In fact, a very recent preprint claims a proof for triangles, resolving a polymath project under the aegis of T.~Tao; see \cite{JM18}, also for more references and an account of the history of the problem.

In this note we introduce an analogue for quantum graphs: to have a clear notion of boundary, we restrict ourselves to trees, where the boundary is the set of vertices of degree one. Then it is to be expected, and we prove, that the extrema of the eigenfunction of the first positive eigenvalue of the Laplacian with standard (or continuity-Kirchhoff) vertex conditions, the natural equivalent of Neumann conditions, are all located at the boundary, establishing a hot spots-type theorem for quantum trees. However, an example shows that these extrema need not be located at maximal (Euclidean) distance from each other within the graph, i.e., the distance between these extremal points may be strictly less than the diameter, contrary to na\"ive intuition. In fact, for the \emph{discrete} Laplacian on a tree, where this problem has also been considered, a similar type of example was found in~\cite{E11}. A more complete analysis of the hot spots of quantum graphs (not just trees) will be given elsewhere.

Let us now briefly summarise some important properties of quantum graphs. We will largely follow the notation of \cite{BK13}, to which we also refer for further details. We will consider finite, compact metric trees $\Gamma$, i.e., the vertex set $V$, the edge set $E$, and the lengths of the edges, are all finite, and $\Gamma$ contains no cycles. We call the \emph{boundary} of $\Gamma$ the set of vertices of degree one. The standard (or Kirchhoff) Laplacian $- \Delta_\Gamma$ is a well-known self-adjoint operator defined on $L^2 (\Gamma)$, which has discrete spectrum of the form $0 = \mu_1 (\Gamma) \leq \mu_2 (\Gamma) \leq \ldots$; see \cite[Sec.~1.4]{BK13} for the particulars. If $\Gamma$ is connected then $\mu_2 (\Gamma)>0$ and $\ker (- \Delta_\Gamma)$ consists of all constant functions on $\Gamma$. It follows that each $f \in \ker (- \Delta_\Gamma - \mu_2 (\Gamma))$ is orthogonal to the constants,
\begin{align}\label{eq:orthogonal}
 \int_\Gamma f \,\textrm{d}x = 0.
\end{align}

\subsection*{Acknowledgements}

The work of J.\,K. was supported by the Funda\c{c}\~{a}o para \linebreak a Ci\^{e}ncia e a Tecnologia, via the program ``Investigador FCT'', reference \linebreak IF/01461/2015, and project PTDC/MAT-CAL/4334/2014.

\section{A hot spots theorem on metric trees}

In this section we state and prove the main result of this note.

\begin{theorem}
Let $\Gamma$ be a finite, compact, connected metric tree and let $f$ be an arbitrary eigenfunction of $- \Delta_\Gamma$ corresponding to the first positive eigenvalue $\mu_2 (\Gamma)$. Then all global minima and maxima of $f$ are located at the boundary of~$\Gamma$.
\end{theorem}

\begin{proof}
Let $f \in \ker (- \Delta_\Gamma - \mu_2 (\Gamma))$ be nontrivial. We are going to show that the global maximum of $f$ lies on the boundary; the statement for the minimum then follows by considering $-f$. For a contradiction, assume that $f$ has a global maximum at an interior point of $\Gamma$, without loss of generality at a vertex $v$ with $\deg (v) \geq 2$. Then all ingoing derivatives of $f$ at $v$ are nonnegative, and due to the Kirchhoff condition it follows that all these derivatives vanish at $v$. On the other hand, $f (v) > 0$ as it is the maximum of the nontrivial function $f$ which satisfies~\eqref{eq:orthogonal}. Let us disconnect $\Gamma$ at $v$ into $\deg (v)$ subgraphs, i.e., we split $\Gamma$ into a new graph having $\deg (v)$ connected components; see Fig.~\ref{fig:split}. 
\begin{figure}[htb]
  \centering
  \begin{tikzpicture}
    \draw[fill] (0,1) circle(0.05);
    \draw[fill] (-1.2,1.8) circle(0.05);
    \draw[fill] (-0.8,0) circle(0.05);
    \draw[fill] (1.5,1) circle(0.05) node[above]{$v$};
    \draw[fill] (1.9,0.2) circle(0.05);
    \draw[fill] (2.5,1.4) circle(0.05);
    \draw[fill] (2.9,1.7) circle(0.05);
    \draw[fill] (3.1,1.5) circle(0.05);
    \draw[fill] (2.8,1.2) circle(0.05);
    \draw (-1.2,1.8)--(0,1)--(-0.8,0);
    \draw (0,1)--(1.5,1); % split here
    \draw (1.5,1)--(1.9,0.2); % split here
    \draw (1.5,1)--(2.5,1.4)--(2.9,1.7);
    \draw (2.5,1.4)--(3.1,1.5);
    \draw (2.5,1.4)--(2.8,1.2);
    \begin{scope}[shift={(7,0)}]
    \draw[fill] (0,1) circle(0.05);
    \draw[fill] (-1.2,1.8) circle(0.05);
    \draw[fill] (-0.8,0) circle(0.05);
    \draw[fill] (1.5,1) circle(0.05);
    \draw (-1.2,1.8)--(0,1)--(-0.8,0);
    \draw (0,1)--(1.5,1); 
    % split here
    \draw[fill] (1.7,0.6) circle(0.05);  
    \draw (1.7,0.6)--(2.1,-0.2); 
    \draw[fill] (2.1,-0.2) circle(0.05);
    % split here
    \draw (1.9,1)--(2.9,1.4)--(3.3,1.7);
    \draw (2.9,1.4)--(3.5,1.5);
    \draw (2.9,1.4)--(3.2,1.2);
    \draw[fill] (1.9,1) circle(0.05);
    \draw[fill] (2.9,1.4) circle(0.05);
    \draw[fill] (3.3,1.7) circle(0.05);
    \draw[fill] (3.5,1.5) circle(0.05);
    \draw[fill] (3.2,1.2) circle(0.05);
    \end{scope}
  \end{tikzpicture}
  \caption{The tree $\Gamma$ before and after splitting at the vertex $v$.}
  \label{fig:split}
\end{figure}
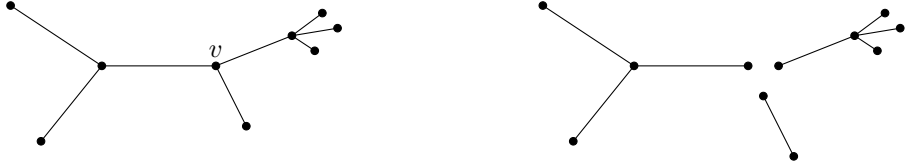
Let $\hat \Gamma$ be one of these components. By the above observations, $\hat f := f|_{\hat \Gamma}$ is not identically zero on~$\hat \Gamma$ and has a vanishing derivative at $v$; as $- \hat f'' = \mu_2 (\Gamma) \hat f$ holds edgewise on $\hat \Gamma$, it follows that $\hat f$ is an eigenfunction of $- \Delta_{\hat \Gamma}$ corresponding to the eigenvalue $\mu_2 (\Gamma)$. In particular, $\mu_2 (\hat \Gamma) \leq \mu_2 (\Gamma)$. On the other hand, $\Gamma$ can be obtained from $\hat \Gamma$ by gluing pendant trees to the vertex $v$ of $\hat \Gamma$ so that actually $\mu_2 (\hat \Gamma) = \mu_2 (\Gamma)$ by the domain monotonicity principle~\cite[Thm.~2]{KMN13}. Hence $\hat f$ is an eigenfunction of $- \Delta_{\hat \Gamma}$ corresponding to $\mu_2 (\hat \Gamma)$ which is nonzero at the splitting vertex $v$, and thus~\cite[Thm.~2]{KMN13} even yields $\mu_2 (\Gamma) < \mu_2 (\hat \Gamma)$, a contradiction. 
\end{proof}

\section{An example}

We now give an example showing that the global maxima and minima of the eigenfunction $f$ corresponding to $\mu_2 (\Gamma)$ need not be at the points realising the diameter of a tree, but may be closer together. To construct this example, we start out with a path graph (interval) $\mathcal{P}$ of length $1$ and a star graph $\mathcal{S}$ consisting of three edges, each of length $\frac{1}{2}-\varepsilon$ for some $\varepsilon \geq 0$ to be chosen later, meeting at a central vertex, as depicted in Fig.~\ref{fig:example} (left). We form $\Gamma$ by first reflecting two copies of each of them, to form $\mathcal{P}^2$ and $\mathcal{S}^2$, respectively, and then gluing these together at a central vertex $v_0$ as shown in Fig.~\ref{fig:example}.
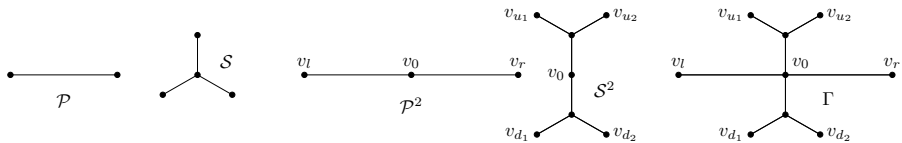
\begin{figure}[htb]
  \centering
  \resizebox{12cm}{!}{
  \begin{tikzpicture}
    %path graph
    \draw[fill] (0,0) circle(0.05);
    \draw[fill] (2,0) circle(0.05);
    \draw (0,0) -- (2,0);
    \node at (1,-0.25) [anchor=north] {$\mathcal{P}$};
    %star graph
    \draw[fill] (3.5,0) circle(0.05);
    \draw[fill] (3.5,0.75) circle(0.05);
    \draw[fill] (2.85,-0.375) circle(0.05);
    \draw[fill] (4.15,-0.375) circle(0.05);
    \draw (3.5,0) -- (3.5,0.75);
    \draw (3.5,0) -- (2.85,-0.375);
    \draw (3.5,0) -- (4.15,-0.375);
    \node at (3.8,0) [anchor=south west] {$\mathcal{S}$};
    %reflected copies
    \begin{scope}[shift={(6.5,0)}]
    \draw[fill] (-1,0) circle(0.05);
    \draw[fill] (1,0) circle(0.05);
    \draw[fill] (3,0) circle(0.05);
    \draw (-1,0) -- (3,0);
    \draw[fill] (4,0) circle(0.05);
    \draw[fill] (4,0.75) circle(0.05);
    \draw[fill] (4,-0.75) circle(0.05);
    \draw[fill] (3.35,1.125) circle(0.05);
    \draw[fill] (4.65,1.125) circle(0.05);
    \draw[fill] (3.35,-1.125) circle(0.05);
    \draw[fill] (4.65,-1.125) circle(0.05);
    \draw[fill] (4,0) -- (4,0.75);
    \draw[fill] (4,0) -- (4,-0.75);
    \draw[fill] (4,0.75) -- (3.35,1.125);
    \draw[fill] (4,0.75) -- (4.65,1.125);
    \draw[fill] (4,-0.75) -- (3.35,-1.125);
    \draw[fill] (4,-0.75) -- (4.65,-1.125);
    \node at (1,0) [anchor=south] {$v_0$};
    \node at (4,0) [anchor=east] {$v_0$};
    \node at (-1,0) [anchor=south] {$v_l$};
    \node at (3,0) [anchor=south] {$v_r$};
    \node at (3.35,1.125) [anchor=east] {$v_{u_1}$};
    \node at (4.65,1.125) [anchor=west] {$v_{u_2}$};
    \node at (3.35,-1.125) [anchor=east] {$v_{d_1}$};
    \node at (4.65,-1.125) [anchor=west] {$v_{d_2}$};
    \node at (1,-0.3) [anchor=north] {$\mathcal{P}^2$};
    \node at (4.3,0) [anchor=north west] {$\mathcal{S}^2$};
    \end{scope}
    %example graph
    \begin{scope}[shift={(14.5,0)}]
    \draw[fill] (0,0) circle(0.05);
    \draw[fill] (-2,0) circle(0.05);
    \draw[fill] (2,0) circle(0.05);
    \draw[fill] (0,0.75) circle(0.05);
    \draw[fill] (0,-0.75) circle(0.05);
    \draw[fill] (-0.65,1.125) circle(0.05);
    \draw[fill] (0.65,1.125) circle(0.05);
    \draw[fill] (-0.65,-1.125) circle(0.05);
    \draw[fill] (0.65,-1.125) circle(0.05);
    \draw[fill] (0,0) -- (-2,0);
    \draw[fill] (0,0) -- (2,0);
    \draw[fill] (0,0) -- (0,0.75);
    \draw[fill] (0,0) -- (0,-0.75);
    \draw[fill] (0,0.75) -- (-0.65,1.125);
    \draw[fill] (0,0.75) -- (0.65,1.125);
    \draw[fill] (0,-0.75) -- (-0.65,-1.125);
    \draw[fill] (0,-0.75) -- (0.65,-1.125);
    \node at (0,0) [anchor=south west] {$v_0$};
    \node at (-2,0) [anchor=south] {$v_l$};
    \node at (2,0) [anchor=south] {$v_r$};
    \node at (-0.65,1.125) [anchor=east] {$v_{u_1}$};
    \node at (0.65,1.125) [anchor=west] {$v_{u_2}$};
    \node at (-0.65,-1.125) [anchor=east] {$v_{d_1}$};
    \node at (0.65,-1.125) [anchor=west] {$v_{d_2}$};
    \node at (0.8,-0.2) [anchor=north] {$\Gamma$};
    \end{scope}
  \end{tikzpicture}
  }
  \caption{The path graph $\mathcal{P}$ and the star $\mathcal{S}$ (left); the graphs $\mathcal{P}^2$ and $\mathcal{S}^2$ formed by reflecting them (centre); the graph $\Gamma$ formed by gluing the reflections together at the central vertex $v_0$ (right). Here $l$ stands for ``left'', $r$ for ``right'', $u$ for ``up'' and $d$ for ``down''.}
  \label{fig:example}
\end{figure}
Then $\diam(\Gamma) = 2$ is realised only by the path joining $v_l$ and $v_r$ provided $\varepsilon>0$. But we claim that for sufficiently small $\varepsilon>0$, $\mu_2 (\Gamma)$ is simple and its eigenfunction $f$ vanishes identically on $\mathcal{P}^2 \subset \Gamma$, being supported on $\mathcal{S}^2 (\Gamma)$ and without loss of generality reaching its maximum at the vertices $v_{u_1}$, $v_{u_2}$ and minimum at the vertices $v_{d_1}$, $v_{d_2}$. Indeed, by standard arguments using the symmetries of $\Gamma$, the eigenfunctions of $-\Delta_\Gamma$ can be chosen to form an orthonormal basis of $L^2 (\Gamma)$, such that each eigenfunction $f$ falls into exactly one of the following categories:

1. $f(v_0)=0$ and $f$ is supported on $\mathcal{P}^2$, corresponding to an eigenfunction and eigenvalue of $-\Delta_{\mathcal{P}^2}$;

2. $f(v_0)=0$ and $f$ is supported on $\mathcal{S}^2$, corresponding to an eigenfunction and eigenvalue of $-\Delta_{\mathcal{S}^2}$;

3. $f(v_0) \neq 0$ and $f$ is supported on the whole of $\Gamma$.

\noindent The smallest eigenvalue in the first case is $\mu_2 (\mathcal{P}^2) = \pi^2/4$, while the smallest in the second is $\mu_2 (\mathcal{S}^2)$; both of these are seen to be simple, and it may be checked that the eigenfunction of the latter reaches its extrema at the degree-one vertices of $\mathcal{S}^2$. Now if $\varepsilon=0$, then $\mu_2 (\mathcal{S}^2) < \mu_2 (\mathcal{P}^2)$ by \cite[Thm.~2]{KMN13}. Hence, by continuity of $\mu_2$ with respect to edge lengths, cf.~\cite[Sec.~3.1.2]{BK13}, $\mu_2 (\mathcal{S}^2) < \mu_2 (\mathcal{P}^2)$ still holds if $\varepsilon>0$ is small enough; fix any $\varepsilon>0$ with this property. In the third case, the smallest eigenvalue is $0 = \mu_1 (\Gamma)$. But if $f$ is any non-constant eigenfunction in this class with eigenvalue $\mu$, then since $f$ is sign-changing by \eqref{eq:orthogonal} and $f(v_0)\neq 0$, at least one of its nodal domains (connected components of $\{f\neq 0\}$), call it $\mathcal{N}$, must be a proper subset of one of the copies of $\mathcal{P}$ or $\mathcal{S}$, say $\mathcal{S}$. Then by standard arguments $\mu = \lambda_1 (\mathcal{N})$, the first eigenvalue of the Laplacian with a Dirichlet condition at $\partial{\mathcal{N}}:=\overline{\mathcal{N}} \cap (\Gamma\setminus \mathcal{N})$ and standard conditions elsewhere. But then $\lambda_1 (\mathcal{N}) > \lambda_1 (\mathcal{S})$ by a Dirichlet version of the strict domain monotonicity principle, since $\mathcal{N} \subsetneq \mathcal{S}$ (where $\mathcal{S}$ is equipped with a Dirichlet condition at $v_0$), cf.~\cite[Thm.~3.10]{BKKM18}. On the other hand, $\lambda_1 (\mathcal{S}) = \mu_2 (\mathcal{S}^2) < \mu$ is the smallest eigenvalue in the second case; hence no eigenvalue in the third category can equal $\mu_2 (\Gamma)$. An analogous argument yields a similar comparison with the first case if $\mathcal{N}\subset\mathcal{P}$ instead. At any rate, we conclude that $\mu_2 (\Gamma)$ is simple and equals $\mu_2 (\mathcal{S}^2)$ for $\varepsilon>0$ small enough, and the unique eigenfunction reaches its maximum at $v_{u_1}$, $v_{u_2}$ and minimum at $v_{d_1}$, $v_{d_2}$. More sophisticated variants of this example, such as where the distance between maximal and minimal points of $f$ can be arbitrarily small, will be discussed at a later point.

\vspace{\baselineskip}

\end{document}